\documentclass[10pt]{amsart}
\usepackage{amsmath,fullpage}
\usepackage{amsthm}
\usepackage{amsfonts}
\usepackage{amssymb}
\usepackage{amscd}
\usepackage[all]{xy}

\usepackage{color}

\usepackage{enumerate}

\usepackage{hyperref}
\hypersetup{
    colorlinks=true,       
    linkcolor=blue,          
    citecolor=blue,        
    filecolor=blue,      
    urlcolor=blue           
}

    \newtheorem{Lem}{Lemma}[section]
    \newtheorem{Lem-Def}{Lemma-Definition}[section]
    \newtheorem{Prop}[Lem]{Proposition}
    
    \newtheorem*{thm}{Theorem}

    \newtheorem{Thm}[Lem]{Theorem}
    \newtheorem{Cor}[Lem]{Corollary}

\theoremstyle{definition}

    \newtheorem{Rem}[Lem]{Remark}

\newcommand{\Sym}{ \text{Sym}^{24}}

\newcommand{\co}{\mathbb{C}}

\newcommand{\F}{\mathcal F}

\newcommand{\IP}{\mathcal P}

\newcommand{\D}{\mathcal D}
\newcommand{\col}{\colon}

\newcommand{\Ps}{\mathbb{P}}
\newcommand{\ra}{\rightarrow}

\newcommand{\bpd}{{\mathbb{P}^2}^\vee}

\newcommand{\fer}{F}
\newcommand{\kk}{K}

\newcommand{\pgl}{{\mathbb{P}GL}}

\address{Marco Pacini, Universidade Federal Fluminense, Rua M. S. Braga, Niter\'oi (RJ) Brazil}
\email{pacini@impa.br, pacini@vm.uff.br}

\address{Damiano Testa, Mathematics Institute, University of Warwick, Coventry, CV4 7AL, United Kingdom}
\email{adomani@gmail.com}

\begin{document}

\title[Plane quartics with at least 8 hyperinflection points]{Plane quartics with at least 8 hyperinflection points}

\author[Marco Pacini and Damiano Testa]{Marco Pacini and Damiano Testa}

\begin{abstract}
A recent result shows that a general smooth plane quartic can be recovered from its 24 inflection lines and a single inflection point. Nevertheless, the question whether or not a smooth plane curve of degree at least 4 is determined by its inflection lines is still open. Over a field of characteristic 0, we show that it is possible to reconstruct any smooth plane quartic with at least 8 hyperinflection points by its inflection lines. Our methods apply also in positive characteristic, where we show a similar result, with two exceptions in characteristic 13.
%
%
\end{abstract}

\thanks{The first author was partially supported by CNPq, processo 300714/2010-6.}

\maketitle

\section*{Introduction}

Many questions on plane curves ask for a reconstruction procedure in terms of linear data.  For instance, Caporaso and Sernesi  in~\cite{CS} and Lehavi in~\cite{L} answer the question of whether the configuration of bitangent lines to a smooth plane quartic determines the plane quartic itself.  Similarly, over finite fields, Bogomolov, Korotiaev and Tschinkel in~\cite{BKT}, address the question of determining curves from their ``inflection data''.  More recently, we studied in~\cite{PT} the reconstruction problem for plane curves from their inflection lines: we showed that every smooth plane cubic over a field of characteristic different from 2, can be recovered from its 9 inflection lines, and that a general smooth plane quartic over a field of characteristic 0 can be recovered from its 24 inflection lines and a single inflection point.  Nevertheless, the question whether or not a smooth plane curve of degree at least 4 is determined by its inflection lines is still open. In characteristic 0, we show in this paper that it is possible to reconstruct any smooth plane quartic with at least 8 hyperinflection points by its inflection lines.  In work in progress~\cite{APT} with Abreu, we use the results of this paper to show that a general plane quartic can be reconstructed from the configuration of its inflection lines.

In this paper we reconstruct the smooth plane quartics contained in an explicit family that we call {\em{Vermeulen's list}} (see~\eqref{verli}).  

\begin{thm}
If $C_1,C_2$ are quartics in Vermeulen's list with the same configuration of inflection lines, then $C_1$ and $C_2$ coincide, with the following exceptions in characteristic 13:
\begin{itemize}
\item if $C_1$ is isomorphic to the plane quartic with equation
\[
\begin{array}{rcl}
x^4+y^4+z^4+3 (x^2y^2+x^2z^2+y^2z^2) & = & 0, 
\end{array}
\]
 then there are exactly 3 possibilities for $C_2$;
\item if $C_1$ is isomorphic to the plane quartic with equation
\[
\begin{array}{rcl}
x^4 - y^4 + z^4 + 2 x^2 y^2 + 4 x y z^2 & = & 0,
\end{array}
\]
 then there are exactly 2 possibilities for $C_2$.
\end{itemize}
In all cases the curves $C_1$ and $C_2$ are isomorphic.
\end{thm}

Vermeulen showed in~\cite{V} that over a field of characteristic 0 the quartics in Vermeulen's list are exactly the smooth plane quartics with at least 8 hyperinflection points.  Combining Vermeulen's result with our reconstruction procedure, we deduce that, over a field of characteristic 0, all smooth plane quartics with at least 8 hyperinflection points are determined uniquely by the configuration of their inflection lines.  It follows from Vermeulen's classification that every smooth plane quartic with at least 8 hyperinflection points has non-trivial automorphism group: the presence of non-trivial automorphisms is an important ingredient in our reconstruction result.

Our methods work with very minor restrictions on the characteristic of the ground field and in fact allow us to determine exceptions to our main statement over fields of positive characteristic.  A curious and unexpected consequence of our work is that all the exceptions that we found are defined over fields of characteristic~$13$: we do not have an explanation of the special nature of this characteristic.

We want to make one final comment on the characteristics $2$ and $3$.  We always exclude fields of characteristic $2$ since all our models are singular in this case.  We exclude fields of characteristic $3$ since we believe that a smooth plane quartic over a field of characteristic $3$ has either inseparable Gauss map (for instance, the curves $\fer,\kk,V$ of Vermeulen's list) or all inflection points of multiplicity at least 3 (for instance, the general curve of the family $V_u$ of Vermeulen's list): in the first case, the knowledge of the inflection lines is equivalent to the knowledge of the dual curve and hence of the curve itself; in the second case, we found a one-parameter family of curves that are not reconstructed by the configuration of their inflection lines.  We do not analyze the case of characteristic $3$ further, even though we expect that our techniques combined with the methods developed by Hefez and Kleiman in~\cite{HK} should be applicable in this context.

As a consequence of our method, we obtain a description of the automorphism groups of the curves in Vermeulen's list valid in all admissible characteristics.  These groups coincide with the groups in characteristic~0 that can be found for instance in Dolgachev's book~\cite{Do}.

Cumino, Esteves and Gatto give an intrinsic description of families and limit of hyperinflection points of curves in~\cite{CEG}; their terminology differs from ours since they call {\emph{special Weierstrass points}} what we call {\emph{hyperinflection points}}.

\subsection*{Acknowledgments} The authors wish to thank the organizers of the 12th ALGA meeting for the inspiring conference where this collaboration started. The authors also want to thank the organizers of ColGA for their warm and kind hospitality, and IMPA and the mathematical departments of UFF and UFRJ for the stimulating environments. 



\section{Inflection lines of smooth plane quartics}

Let $k$ be an algebraically closed field of any characteristic.  We determine a sufficient condition to reconstruct a smooth plane curve within its $\pgl_3(k)$-orbit via its inflection lines. Before, we prove an easy result on automorphisms of $\Ps^2$.

\begin{Lem}\label{aut-id}
If an automorphism of $\Ps^2$ fixes at least 6 inflection points of a smooth plane quartic, then it is the identity.
\end{Lem}

\begin{proof}
Let $C\subset\Ps^2$ be a smooth plane quartic and let $\IP\subset\Ps^2$ be a set of inflection points of $C$ of cardinality~$6$. Assume that $\sigma\in \pgl_3(k)$ is an automorphism of $\Ps^2$ fixing the points of $\IP$. Since $C$ intersects a line of $\Ps^2$ in a set of cardinality at most 4 and the cardinality of $\IP$ is 6, it follows that $\IP$ is not contained in a line of $\Ps^2$.  In particular, there are 3 points $P_1,P_2,P_3$ of $\IP$ that are  in general position. Without loss of generality,  we can assume that they are the points $P_1=[1,0,0]$, $P_2=[0,1,0]$, $P_3=[0,0,1]$. Since $\sigma$ fixes $P_1,P_2,P_3$,  it follows that $\sigma$ is induced by a diagonal matrix  $diag(\lambda_1,\lambda_2,\lambda_3)$, where $\lambda_1,\lambda_2,\lambda_3\in k^\times$. Let $\ell_{ij}$ be the line of $\Ps^2$ containing $P_i$ and $P_j$, for $i,j$ distinct indices in $\{1,2,3\}$.  If $\IP$ is not contained in the union $\ell_{12}\cup \ell_{13}\cup \ell_{23}$, then $\sigma$ fixes $4$ points in general position, and hence it is the identity. On the other hand, if $\IP$ is contained in $\ell_{12}\cup \ell_{13}\cup \ell_{23}$, then there are two of the $\ell_{ij}$'s containing each one at least 3 points of $\IP$, and hence the two lines are fixed pointwise by $\sigma$. 
 It follows that $\lambda_1=\lambda_2=\lambda_3$, and hence $\sigma$ is the identity.
\end{proof}

\begin{Prop}\label{recon-orbit}
Let $C$ be a smooth plane quartic with at least $6$ inflection points.  The automorphism group of $C$ injects in the group of linear automorphisms of the configuration of its inflection lines. Furthermore, if this injection is an isomorphism and $C'$ is a plane quartic isomorphic to $C$ but different from $C$, then the configurations of inflection lines of $C$ and $C'$ are different.
\end{Prop}

\begin{proof}
Let $G_1$ be the automorphism group of $C$ and $G_2 \subset \pgl_3 (k)$ be the group of linear automorphisms of the configuration of inflection lines of $C$. There is a natural group homomorphism  $\varphi \col G_1 \ra G_2$ defined as follows.  If $\tau\in G_1$ is induced by an automorphism  $\sigma\in \pgl_3(k)$, then $\sigma$ permutes the inflection lines of $C$ and we set $\varphi(\tau) = \sigma $. Notice that if $\varphi(\tau)=id$ and $\tau$ is induced by  $\sigma\in \pgl_3(k)$, then $\sigma$ fixes the inflection lines of $C$, and hence also the inflection points of $C$.  It follows from Lemma \ref{aut-id} that $\sigma$, and hence $\tau$, is the identity, implying that $\varphi$ is an injection.

Let us show the second sentence of the lemma.  Assume that the plane quartic $C'$ is isomorphic to $C$ and different from $C$. Write $C\ne C'=\sigma(C)$, for some $id\ne \sigma\in  \pgl_3(k)$. In particular, $\sigma|_{C}$ is not an automorphism of $C$. On the other hand, if the configurations of inflection lines of $C$ and $C'$ are equal, then $\sigma$ induces a non-trivial linear automorphism of the configuration of inflection lines of $C$. It follows that the homomorphism $\varphi\col G_1\ra G_2$ is not surjective, and we are done.
\end{proof}

\begin{Rem}
In Theorem~\ref{aut-iso} we show that, over fields of characteristic different from $2$ and $3$, the injection constructed in Proposition~\ref{recon-orbit} is an isomorphism for any smooth plane quartic in Vermeulen's list~\eqref{verli} with two exceptions in characteristic~$13$ (see Remarks~\ref{Kecce} and~\ref{eccezioni}).  If the characteristic of the field $k$ is 0, for all the smooth plane quartics we considered, the group of projective automorphisms of the configuration of inflection lines coincides with the automorphism group of the curve.
%
Nevertheless, there exist singular plane quartics for which the injection is not an isomorphism. For example, let $C$ be an irreducible plane quartic with exactly 3 cusps. Up to an isomorphism, the equation of the curve $C$ is 
\[
\begin{array}{lll}
C \colon \quad x^2y^2+x^2z^2+y^2z^2-2xyz(x+y+z) &  = & 0
\end{array}
\]
(see for example~\cite{BG}). Let $\F_C\in \Sym(\bpd)$ be the cycle consisting of the 3 cuspidal tangents of $C$ each one appearing with multiplicity 8.  The map 
\[
\F\colon \Ps^{14}\dashrightarrow \Sym(\bpd)
\]
sending a smooth plane quartic to its cycle of inflection lines extends over $C$ and sends $C$ to the cycle $\F_C$ (see~\cite[Propositions~2.5 and~2.10]{PT}). Thus, we can regard $\F_C$ as the configuration of inflection lines of $C$. Let $G_1$ and $G_2$ be respectively the automorphism group of $C$ and the group of linear automorphism of the cuspidal tangents of $C$. The group $G_1$ is finite (and it is isomorphic to the symmetric group on 3 elements, acting as permutation of the coordinates). The cuspidal tangents of $C$ are the lines $x=y$, $x=z$ and $y=z$, and hence  the group $G_2$ is isomorphic to $\co^*\times \co^*$. The natural group homomorphism $\varphi\col G_1\ra G_2$ described in the proof of Lemma~\ref{recon-orbit} is an injection, but not an isomorphism, since $G_1$ is finite, while $G_2$ is infinite. In particular, there are automorphism of $\Ps^2$ fixing the configuration $\F_C$ and not fixing $C$. Thus, although the map $\F$ is generically injective (see~\cite[Lemma~3.2]{PT}), the fiber $\F$ passing through $C$ is positive dimensional. 
\end{Rem}

\section{Quartics with at least 8 hyperinflection points}


We begin this section with a list of smooth plane quartics that we call {\em{Vermeulen's list}}.  In Theorem~\ref{aut-iso} we show that the configuration of inflection lines of the quartics in Vermeulen's list uniquely characterizes the quartics within the list.  It follows 
from~\cite{V} that, over a field of characteristic 0, Theorem~\ref{aut-iso} is enough to reconstruct the quartics in Vermeulen's list among all smooth plane quartics (Corollary~\ref{main}).



\subsubsection*{Vermeulen's list}
We call {\em{Vermeulen's list}} the family of smooth plane quartics isomorphic to the curves with equations 
\begin{equation} \label{verli}
\left\{
\begin{array}{lrcl}
\fer \colon & x^4+y^4+z^4 & = & 0 , 
\\[7pt]
\kk \colon & x^4+y^4+z^4+3 (x^2y^2+x^2z^2+y^2z^2) & = & 0 , 
\\[7pt]
C_\pm \colon & (21 \pm 8 \sqrt{7}) z^4 - 6 (2 \pm \sqrt{7}) x y z^2 + (3 \pm \sqrt{7}) (x^3+y^3) z - 3 x^2 y^2 & = & 0, 
\\[7pt]
V \colon & x^4 - 7 y^4 - z^4 - 42 x^2 y^2 + 12 x y z^2 & = & 0 , 
\\[7pt]
V_u \colon & u x^4 + y^4 - z^4 - 2 x^2 y^2 - 4 x y z^2 & = & 0 , 
\end{array}
\right.
\end{equation}
where $\sqrt{7}$ is fixed square root of $7$ in $k$ and $u \in k$ is such that the curve $V_u$ is smooth.  All these equations define curves that are singular if the characteristic of the field $k$ is $2$: we always exclude this possibility.  More precisely, these are the conditions for smoothness of the curves in Vermeulen's list.
\begin{itemize}
\item 
The curve $\fer$ is smooth if and only if the characteristic of $k$ is different from $2$.
\item 
The curve $\kk$ is smooth if and only if the characteristic of $k$ is different from $2$ or $5$.
\item 
The curves $C_\pm$ are smooth if and only if the characteristic of $k$ is different from $2$, $3$ or $7$.
\item 
The curve $V$ is smooth if and only if the characteristic of $k$ is different from $2$ or $7$.
\item 
The curve $V_0$ is singular with the unique singular point $[1,0,0]$; the curve $V_1$ is the union of the two lines with equation $z = \pm(x - y)$ 
and of the singular conic with equation $(x+y)^2 + z^2 = 0$.  For all values of $u$ different from $0,1$, the curve $V_u$ is smooth, if the characteristic of $k$ is different from $2$.
\end{itemize}
In our arguments, we always make the implicit assumption that the curves we are treating are non-singular: this is reflected by a restriction on either the characteristic or the specific member of the family.  In particular, from now on, we assume that the characteristic of the algebraically closed field $k$ is different from $2$.  Moreover, we exclude fields of characteristic $3$; in this case, we found a one-parameter family of counterexamples, outlined in Remark~\ref{eccezioni}.



The curve $V$ is obtained from the curve 
\[
8(x^4 + y^4) + 48 \sqrt{-7} x^2 y^2 + 24 (1-\sqrt{-7}) x y z^2 - (7-3\sqrt{-7}) z^4 = 0
\]
in~\cite[p.~1576]{G} by evaluating the coordinates in $x , \sqrt[4]{-7} y , \frac{2 z}{\sqrt[8]{-7} \sqrt{1-\sqrt{-7}}}$.  Similarly, performing the change 
\[
[x ,y,z] \longmapsto \left[ \frac{x}{t + 1} \; , \; \frac{x - y + z}{2} \; , \; \frac{x - y - z}{2t} \right]
\]
in the coordinates of~\cite[Theorem~4.1]{G}, transforms the family 
\[
(t^2+1) (x^2-y z)^2 = y z (2 x-y-z) (2 t x-y-t^2 z)
\]
to the family 
\[
-\frac{1}{8t} \left( \left( \frac{t - 1}{t + 1} \right)^4 x^4 + y^4 - z^4 - 2 x^2 y^2 - 4 x y z^2 \right) = 0 .
\]
Setting $s = \frac{t - 1}{t + 1}$ and $u = s^4$ we obtain the family of plane quartics $V_u$ appearing in Vermeulen's list.

In~\cite{KK}, Kuribayashi and Komiya proved that a smooth plane quartic with exactly 12 hyperinflection points over a field of characteristic 0 is isomorphic either to the Fermat curve $\fer$ or to the curve $\kk$.

In~\cite{V}, Vermeulen showed that over a field of characteristic 0, up to isomorphism, there are no smooth plane quartics with 10 and 11 hyperinflection points, that the two curves $C_\pm$ are the only quartics with exactly $9$ hyperinflection points (\cite[p.~143, 20.12]{V} or~\cite[p.~1564,~\S3]{G}), and that the curves $V$ and $V_u$ are the only quartics with exactly $8$ hyperinflection points.

Hence, over a field of characteristic 0, Vermeulen's list is the list of all smooth plane quartics with at least $8$ hyperinflection points.

Recall that in this section we always assume that the characteristic of the field $k$ is different from~$2$ and~$3$.

In the proof of Theorem~\ref{aut-iso} we start with a quartic in $\mathbb{P}^2$ and with the corresponding configuration of inflection lines in $\bpd$.  We usually choose coordinates on $\bpd$ dual to the coordinates in $\mathbb{P}^2$, thus identifying $\bpd$ and $\mathbb{P}^2$; we denote the coordinates on $\mathbb{P}^2$ and on $\bpd$ with the same symbols $x,y,z$, since no confusion arises.

\begin{Thm}\label{aut-iso}
If $C_1,C_2$ are quartics in Vermeulen's list with the same configuration of inflection lines, then $C_1$ and $C_2$ coincide, with the following exceptions in characteristic 13:
\begin{itemize}
\item if $C_1$ is isomorphic to $K$, then there are exactly 3 possibilities for $C_2$;
\item if $C_1$ is isomorphic to $V_{-1}$, then there are exactly 2 possibilities for $C_2$.
\end{itemize}
In any case, the curves $C_1$ and $C_2$ are isomorphic.
\end{Thm}

To prove Theorem~\ref{aut-iso} we argue separately for the curves in Vermeulen's list having exactly $12$, $9$ or $8$ hyperinflection lines: this is justified since in the course of our reasoning we shall see that the number of hyperinflection lines of the smooth curves in Vermeulen's list is independent of the characteristic.  Our strategy, then, consists of two steps.  
First, the automorphism group of a quartic in Vermeulen's list is isomorphic to the linear automorphism group of its configuration of inflection lines: applying Proposition~\ref{recon-orbit} we deduce that if two quartics in Vermeulen's list have the same configuration of inflection lines, then they are not isomorphic.  Second, the configurations of inflection lines of two non-isomorphic plane quartics in Vermeulen's list are not conjugate by an element of $\pgl_3(k)$: using the first step, the reconstruction result follows.

\subsection*{Proof of Theorem~\ref{aut-iso} for the quartics $\fer$ and $\kk$}

The inflection points of the curve $\fer$ are the points obtained from $[0,1,\varepsilon]$ by permuting the coordinates and where $\varepsilon$ ranges among the primitive roots of unity of order 8; the inflection line corresponding to the inflection point $[0,1,\varepsilon]$ is the line with coordinates $[0,1,\varepsilon^3]$.  Let $\F \subset \bpd$ denote the reduced set of points corresponding to the inflection lines of the curve $\fer$.  The two polynomials $x y z$ and $x^4+y^4+z^4$ are coprime and vanish on $\F$.  It follows that $V(x y z , x^4+y^4+z^4)$ is a scheme of dimension zero and degree twelve containing $\F$, so that these two schemes coincide.  In particular, the ideal of $\F$ is generated by $x y z$ and $x^4+y^4+z^4$.  The polynomial $x y z$ is, up to scaling, the unique polynomial of degree 3 vanishing on $\F$.  We obtain that the linear automorphism group of $\F$ must be a subgroup of the linear automorphism of $V (x y z)$, and it is therefore contained in the group $G$ of matrices in $\pgl_3(k)$ with exactly one non-zero entry in each row and column.  Note that any element of the group $G$ transforms $x^4+y^4+z^4$ into a polynomial of degree 4 whose monomials are fourth powers, and that $x^4+y^4+z^4$ is, up to scaling, the unique polynomial of degree 4 vanishing on $\F$ whose monomials are fourth powers.  We deduce that the group of linear automorphisms of $\F$ injects in the automorphism group of $V(x^4+y^4+z^4)$ and hence it follows from Proposition \ref{recon-orbit} that the automorphism group of $\fer$ is isomorphic to the linear automorphism group of its configuration of inflection lines.

Denote by $G$ the group of permutations and sign changes of the coordinates of $\mathbb{P}^2$; the group $G$ is isomorphic to the symmetric group $\mathfrak{S}_4$ and we shall see that it coincides with the automorphism group of the curve $\kk$.  The twelve inflection points of the curve $\kk$ form the orbit of the point $[1,1,i]$ under the group $G$.  The coordinates of the inflection lines to the curve $\kk$ form the orbit of the point $[1,1,2i]$ under the group $G$; denote by $\F {\subset \bpd}$ the set of the twelve inflection lines of $\kk$.  Let $\widetilde{G}$ be the group of linear automorphisms of $\bpd$ stabilizing the set $\F$; by abuse of notation, we denote by $G$ also the image of the group $G$ acting on $\bpd$, so that the group $\widetilde{G}$ contains $G$.  We show that the group $\widetilde{G}$ coincides with the group $G$, unless the characteristic of the ground field is $13$: in this case, the index of the group $G$ in $\widetilde{G}$ is $3$.

Denote by $X$ the blow up of $\bpd$ at the points in $\F$: the surface $X$ is a smooth projective rational surface.  Each of the three conics 
\[
\begin{array}{l@{\quad}rcl}
D_1 \colon & 3x^2+y^2+z^2 & = & 0 \\[5pt]
D_2 \colon & x^2+3y^2+z^2 & = & 0 \\[5pt]
D_3 \colon & x^2+y^2+3z^2 & = & 0
\end{array}
\]
contains 8 of the points in $\F$, and each point of $\F$ appears in exactly two of the conics $D_1,D_2,D_3$.  It follows that the union $D \subset X$ of the strict transforms of the three conics $D_1,D_2,D_3$ in $X$ is an effective representative of twice the anticanonical divisor on $X$.  We deduce that the divisor $D$ is smooth and its irreducible components have square ${-4}$; in particular, all the effective sections of the positive multiples of the anticanonical linear system have support contained in $D$.  Therefore, every element of the group $\widetilde{G}$ must also stabilize the union $D_1 \cup D_2 \cup D_3$, and hence induces a permutation of the set $\{D_1,D_2,D_3\}$.  It is immediate to check that the group $G$ acts simply transitively on the set of pairs $(p,C)$, where $p$ is an inflection line of $\kk$ and $C$ is one of the two conics $D_1,D_2,D_3$ containing $p$.  Denote by $G' \subset \widetilde{G}$ the subgroup of $\widetilde{G}$ stabilizing the pair $([1,1,2i],D_1)$.  The index of the group $\widetilde{G}$ in $G$ is $24$, and the group $G$ intersects $\widetilde{G}$ trivially.  To conclude it suffices to compute the group $G'$.  Denote by $G''$ the group of automorphisms of $\bpd$ stabilizing the conic $D_1$ fixing the point $[1,1,2i]$ and stabilizing the intersection $D_1 \cap D_2$; the group $G'$ is a subgroup of $G''$.  Let $E$ be the elliptic curve obtained as a double cover of $D_1$ branched over the intersection $D_1 \cap D_2$ with origin lying above the point $[1,1,2i]$; the group $G''$ is the quotient by the elliptic involution of the automorphism group of the elliptic curve $E$.  Hence, the group of $G''$ is trivial unless $j(E)$ is $1728$ or $0$ and a direct calculation shows that the $j$-invariant of the elliptic curve $E$ is $35152/9$.  Thus, the group $G''$ is trivial unless 
\begin{itemize}
\item 
the characteristic of $k$ is $7$, in which case the $j$-invariant of $E$ is $1728$ and $G''$ has order $2$, or 
\item 
the characteristic of $k$ is $13$, in which case the $j$-invariant of $E$ vanishes and $G''$ has order $3$, 
\end{itemize}
recall that we always exclude the cases in which the characteristic of $k$ is $2$ or $3$ and we also exclude the case of fields of characteristic $5$, since the equation of the curve $\kk$ modulo $5$ defines a singular quartic.

In the case of characteristic $7$, the involution 
\[
\begin{array}{rcl}
\gamma_7 \colon \quad \mathbb{P}^2 & \longrightarrow & \mathbb{P}^2 \\[5pt]
[x,y,z] & \longmapsto & [\frac{1}{2}z , iy , -2x] 
\end{array}
\]
is the unique non-trivial element of $G''$.  Since $\gamma_7$ does not stabilize the set of inflection lines of $\kk$, we deduce that the group $G'$ is trivial and therefore the groups $G$ and $\widetilde{G}$ coincide.

In the case of characteristic $13$, the linear transformation 
\[
\begin{array}{rcl}
\gamma_{13} \colon \quad \mathbb{P}^2 & \longrightarrow & \mathbb{P}^2 \\[5pt]
[x,y,z] & \longmapsto & [z , 3 x , \frac{1}{3} y] 
\end{array}
\]
is a generator of the group $G''$.  In this case, the map $\gamma_{13}$ is an element of the group $\widetilde{G}$ and it is not an automorphism of the curve $\kk$; we conclude that the group $G'$ coincides with the group $G''$ and therefore the group $G$ has index $3$ in the group $\widetilde{G}$.  Thus, with the exception of fields of characteristic $13$, the automorphism group of $\kk$ is isomorphic to the linear automorphism group of its configuration of inflection lines.


Finally, we can distinguish the configurations of inflection lines of the curves $\fer$ and $\kk$ by observing that the inflection lines of the curve $\fer$ are contained in 3 lines, while the inflection lines of the curve $\kk$ are not: the proof of Theorem \ref{aut-iso} is complete for the curves $\fer$ and $\kk$.
\qed

\begin{Rem}\label{Kecce}
Over a field of characteristic 13, the three curves 
\[
\begin{array}{lrcl}
\kk_1 \colon & x^4+y^4+z^4+3 (x^2 y^2+x^2 z^2+y^2 z^2) & = & 0 \\[5pt]
\kk_2 \colon & x^4+3y^4+9z^4+3 (9x^2 y^2+3x^2 z^2+y^2 z^2) & = & 0 \\[5pt]
\kk_3 \colon & x^4+9y^4+3z^4+3 (3x^2 y^2+9x^2 z^2+y^2 z^2) & = & 0
\end{array}
\]
have the same configuration of inflection lines.  The curves themselves are projectively equivalent: the automorphism $\gamma_{13} \colon [x,y,z] \mapsto [z,3x,\frac{1}{3}y]$ of $\mathbb{P}^2$ permutes the curves $\kk_1,\kk_2,\kk_3$ cyclically.
\end{Rem}

\subsection*{The group $\overline{D}$ and the pencil $\mathcal{Q}$}

We introduce now a subgroup $\overline{D}$ of $\pgl_3(k)$ that we will use often in the reconstruction arguments of this section.  The subgroup $\overline{D} \subset \pgl_3(k)$ is the group stabilizing the conic with equation $xy + \alpha z^2 = 0$, for some $\alpha \in k^\times$, and the point $[0,0,1]$.  The group $\overline{D}$ is the group consisting of the linear transformations 
\begin{equation}\label{Dbar}
\overline{D} = \left\{ \left.
\underbrace{\rho_\alpha \colon [x,y,z] \longmapsto \left[ \alpha x , \alpha^{-1} y , z \right]}_{\textup{diagonal}} 
\quad \quad {\textup{and}} \quad \quad 
\underbrace{\sigma_\alpha \colon [x,y,z] \longmapsto \left[ \alpha y , \alpha^{-1} x , z \right]}_{\textup{anti-diagonal}} 
\hspace{10pt} \right| \hspace{10pt} 
\alpha  \in k^\times \right\} ;
\end{equation}
we call {\em{diagonal}} the elements of $\overline{D}$ of the first kind and {\em{anti-diagonal}} the elements of the second kind.  Let 
\begin{equation} \label{penq}
\mathcal{Q} = \left\{ \lambda xy + \mu z^2 = 0 ~\mid~ [\lambda , \mu ] \in \mathbb{P}^1 \right\}
\end{equation}
denote the pencil of conics generated by the conics $xy=0$ and $z^2=0$.  The group $\overline{D}$ can be equivalently described as the group stabilizing two distinct conics in the pencil $\mathcal{Q}$.  Note that every subgroup of $\overline{D}$ either consists entirely of diagonal elements, or it is generated by its diagonal elements and a single anti-diagonal element.

\subsection*{Proof of Theorem~\ref{aut-iso} for the quartics $C_\pm$}


Let $\zeta$ denote a root of unity of order three.  Observe that the subgroup 
$G$ of $\pgl_3(k)$ generated by the linear automorphisms 
\[
[x,y,z] \longmapsto [\zeta x , \zeta^{-1} y , z]
\quad \quad {\textup{and}} \quad \quad 
[x,y,z] \longmapsto [y , x , z]
\]
is isomorphic to the symmetric group $\mathfrak{S}_3$ and is contained in the automorphism group of the two curves $C_\pm$.  Observe that every conic in the pencil~$\mathcal{Q}$ of~\eqref{penq} is invariant under $G$, and hence, every orbit of $G$ is contained in a $G$-invariant conic in the pencil $\mathcal{Q}$.

The group $G$ acts on the two quartics $C_{\pm}$ and hence also on the configuration of their inflection lines.  We begin by showing that the automorphism group $G_+$ of the configuration of inflection lines of the curve $C_+$ coincides with the group $G$; of course everything we say applies equally well to $C_-$.  We already observed that the group $G_+$ contains the group $G$.  The curve $C_+$ has exactly six simple inflection lines, and they correspond to the points in the intersection 
\[
\left\{
\begin{array}{rcl}
4 \sqrt{7} xy + (12 + 5 \sqrt{7})  z^2 & = & 0 \\[5pt]
28 (x^3 + y^3) + (218 + 45 \sqrt{7}) z^3 & = & 0 ;
\end{array}
\right.
\]
in particular, any element of the group $G_+$ induces a permutation of the six simple inflection lines of $C_+$ and is uniquely determined by the resulting permutation.  Moreover, since the group $G_+$ preserves the conic containing the six simple inflection lines, we deduce that the only element of $G_+$ fixing at least three simple inflection lines is the identity.  Thus the group $G_+$ is isomorphic to a subgroup of the symmetric group $\mathfrak{S}_6$, it contains the group $G$ isomorphic to $\mathfrak{S}_3$, and its order is not divisible $9$, since otherwise $G_+$ would contain a $3$-Sylow subgroup of $\mathfrak{S}_6$ and hence also a $3$-cycle.  The $9$ hyperinflection lines of $C_+$ decompose into two orbits under group $G$, one corresponding to the $6$ points defined by the equations 
\[
\left\{
\begin{array}{rcl}
xy + (8 + 3 \sqrt{7}) z^2 & = & 0 , \\[5pt]
x^3 + y^3 - (25 + 9 \sqrt{7}) z^3 & = & 0
\end{array}
\right.
\]
and the other one corresponding to the three points $[1,1,-1] , [\zeta , \zeta^{-1} , -1] , [\zeta^{-1} , \zeta , -1]$.  If the group $G_+$ did not preserve this $6+3$ decomposition of the hyperinflection lines, then it would act transitively on the nine hyperinflection lines and its order would be divisible by $9$, contrary to what we showed before.  We deduce that the group $G_+$ stabilizes also the conic with equation $xy + (8 + 3 \sqrt{7}) z^2 = 0$ and hence it stabilizes the pencil $\mathcal{Q}$ of~\eqref{penq}.  It follows that $G_+$ is a subgroup of the group $\overline{D}$ of~\eqref{Dbar}, and since it stabilizes the set $\{ [1,1,-1] , [\zeta , \zeta^{-1} , -1] , [\zeta^{-1} , \zeta , -1] \}$ we conclude that $G_+$ coincides with the group $G$, as required.

To finish the proof of Theorem  \ref{aut-iso} for quartics with 9 hyperinflection points, it remains to show that the configuration of inflection lines of the quartics $C_\pm$ are not conjugate.  From the previous argument we deduce that the automorphism groups of the configuration of inflection lines of the two quartics $C_+$ and $C_-$ coincide with the group $G$, so that any linear map transforming the configuration of inflection lines of $C_+$ into the configuration of inflection lines of $C_-$ must normalize the group $G$.  An easy computation shows that the normalizer $N (G)$ of the group $G$ in $\pgl_3(k)$ is the group generated by $G$ itself and the projective matrices of the form $diag(1,1,\lambda)$ for $\lambda \in k^\times$.  Since the only elements in $N(G)$ actually stabilizing the configuration of inflection lines of the curve $C_+$ (or the curve $C_-$) are the elements of $G$, we conclude that the two configurations are not projectively equivalent and the proof is complete.
\qed

\subsection*{Proof of Theorem~\ref{aut-iso} for the quartics $V$ and $V_u$}

Let $i$ denote a square root of $-1$ and recall that the parameter $s$ satisfies the identity $s^4 = u$; observe that the linear transformations 
\[
\rho_i = \rho \colon [x,y,z] \longmapsto [ix , -iy , z] \quad \quad {\textup{and}} \quad \quad \sigma_s \colon [x,y,z] \longmapsto \left[ s^{-1}y , sx , z \right]
\]
induce automorphisms of the curve $V_{s^4}$ and generate a group $D_{4,u}$ isomorphic to the dihedral group ${\rm D}_4$.  Let $\alpha \in k$ denote an element satisfying the identity $\alpha^4 = -7$; in~\S\ref{D13V} and~\S\ref{D13} we show that the groups $D_{4,\alpha}$ and $D_{4,u}$ are the full automorphism groups of the curve~$V$ and of the smooth curves in the family~$V_u$.


As a consequence of the classification of the quartics $\mathcal C$ with 8 hyperinflection points, it follows that  the automorphism group $\mathcal D$ of $\mathcal C$ is isomorphic to the dihedral group ${\rm D}_4$. 
An easy computation shows that the two reducible conics with equation $y^2 = \pm s^2 x^2$, where $s^4 = u$ if $\mathcal{C}$ is the curve $V_u$ and $s = \alpha$ if $\mathcal{C}$ is the curve $V$, and every conic of the pencil $\mathcal Q$ generated by $xy$ and $z^2$ are all the invariant conics under the action of the group $\D$. It follows that every orbit in $\Ps^2$ or $\bpd$ under the action of $\mathcal D$ is contained in a conic of $\mathcal Q$.  In particular, this applies to the orbits of the action of $\mathcal D$  on the set of inflection points or inflection lines of the curves $\mathcal C$. This justifies why so many conics appear in our computations.

\subsubsection{The quartic $V\colon x^4 - 7 y^4 - z^4 - 42 x^2 y^2 + 12 x y z^2 = 0$}\label{D13V}

It is easy to show that the hyperinflection lines of $V$ are defined by the equations
\[
\left\{
\begin{array}{rrcl}
C_1\col & (7+\sqrt{-7})z^2 & =& 4xy \\ 
& y^2 & = & \sqrt{-7}x^2
\end{array}
\right. 
\quad \quad \text{ and } \quad \quad 
\left\{
\begin{array}{rrcc}
C_2\col & (7-\sqrt{-7})z^2 & = & 4xy \\ 
& y^2 & = & -\sqrt{-7}x^2.
\end{array}
\right.
\]

 The quartics 
\[
q = \bigl( (7+\sqrt{-7})z^2 - 4xy \bigr) \bigl( y^2 + \sqrt{-7}x^2 \bigr) 
\quad \quad {\textup{and}} \quad \quad
\overline{q} = \bigl( (7-\sqrt{-7})z^2 - 4xy \bigr) \bigl( y^2 - \sqrt{-7}x^2 \bigr)
\]
vanish on the $8$ hyperinflection lines and therefore also the reducible quartics 
\[
- \frac{(7 - \sqrt{-7}) q + (7 + \sqrt{-7}) \overline{q}}{56} = y (x^3+xy^2-2yz^2)
\quad {\textup{and}} 
\quad 
- \frac{(7 - \sqrt{-7}) q - (7 + \sqrt{-7}) \overline{q}}{8\sqrt{-7}} = x (7x^2y-14xz^2-y^3)
\]
vanish on the hyperinflection lines.  Since the lines $x=0$ and $y=0$ contain no hyperinflection line, it follows that the $8$ hyperinflection lines correspond to $8$ of the base points of the pencil of cubics generated by $x^3+xy^2-2yz^2 = 0$ and $7x^2y-14xz^2-y^3 = 0$.  The point $[0,0,1]$ is a base point of the pencil, but does not correspond to a hyperinflection line.  Moreover, since any cubic vanishing on the $8$ hyperinflection lines, must also vanish on the point $[0,0,1]$, we deduce that the ideal of homogeneous polynomials vanishing on the configuration of hyperinflection lines vanishes on $[0,0,1]$ and therefore every automorphism of the configuration must preserve the point $[0,0,1]$.  The configuration of simple inflection lines is defined by the equations 
\[
x y + z^2 = 0 
\quad \quad {\textup{and}} \quad \quad 
7 x^4 + 362 x^2 y^2 - y^4 = 0 .
\]
We obtain that the automorphism group of the inflection lines must preserve the conic with equation ${xy+z^2 = 0}$ and the point $[0,0,1]$: it is contained in the group $\overline{D}$ of~\eqref{Dbar}.  In particular, the group of automorphisms of the configuration stabilizes each conic in the pencil $\mathcal{Q}$ and the configuration of hyperinflection lines: it must therefore stabilize the intersection of the conic $C_1$ with the configuration of hyperinflection lines.  It is now immediate to check that the subgroup of $\overline{D}$ with these properties is the group $D_{4,\alpha}$, as required. We conclude that the automorphism group of the quartic $V$ is isomorphic to the linear automorphism group of its configuration of inflection lines.

We note that there are no conics containing the 8 hyperinflection points of $V$. Indeed, since the 8 hyperinflection lines are stable under the action of the group $D_{4,\alpha}$, if there were a conic $C$ containing the 8 hyperinflection lines of $V$, then the conic would be invariant under the action of this group and it would be in the pencil $\mathcal{Q}$. Since the conics $C_1$ and $C_2$ are invariant and contain each 4 hyperinflection lines of $V$, the conic $C$ would be equal to both $C_1$ and $C_2$, which is a contradiction.

\subsubsection{The quartics in the family $V_u \colon u x^4 + y^4 - z^4 - 2 x^2 z^2 - 4 x y^2 z = 0$}\label{D13}

In the family $V_u$, the curve $V_0$ is singular and therefore we exclude the value $0$ for $u$.  A straightforward computation shows that the hyperinflection points of the curves $V_u$ lie on the two lines with equations $x=0$ and $y=0$ and that the coordinates of the hyperinflection lines satisfy the equations 
\begin{equation} \label{vergii}
\left\{
\begin{array}{rrcl}
C_h \colon & z^2 & = & - x y \\[5pt]
Q_h \colon & (u + 1) z^4 & = & x^4 + u y^4 .
\end{array}
\right.
\end{equation}
Similarly, we find that the simple inflection points of the curves $V_u$ lie on the conic $2 x y + 3 z^2 = 0$, and that the coordinates of the corresponding inflection lines satisfy the equations 
\begin{equation} \label{vergis}
\left\{
\begin{array}{rrcl}
C_s \colon & (27 u + 5) z^2 & = & - 2^5 x y \\[5pt]
Q_s \colon & (1215 u^2 - 190 u - 1) z^4 & = & 2^9 (x^4 + u y^4) .
\end{array}
\right.
\end{equation}
Let $C_h$ be the conic containing the $8$ hyperinflection lines and let $C_s$ be the conic containing the $8$ simple inflection lines.  Every automorphism of $\mathbb{P}^2$ stabilizing the configuration of inflection lines must stabilize individually the two conics $C_h$ and $C_s$, and the pencil generated by the conics $C_h$ and $C_s$ is the pencil $\mathcal Q$ generated by $xy=0$ and $z^2=0$.

The group preserving the conics $C_h$ and $C_s$ is the group $\overline{D}$ of~\eqref{Dbar}.
To determine the automorphism group of the configuration it is therefore enough to check which of the elements of $\overline{D}$ preserve the configuration; since the anti-diagonal element $\sigma_s$ preserves the configuration, we only need to determine the diagonal automorphisms.  Note that the diagonal elements of $\overline{D}$ preserve the monomials $xy$, $z^2$, $x^4$, $y^4$, $z^4$ up to scaling; in particular, if an element of $\overline{D}$ preserves the scheme defined by the equations~\eqref{vergii} or~\eqref{vergis}, it must preserve the individual equations, up to scaling, since every non-zero multiple of the equations of $C_h$ or $C_s$ involves a monomial divisible by $xy$. A simple computation shows that the equations $Q_h$ and $Q_s$ are preserved up to scaling by the diagonal element $\rho_\zeta$ of $\overline{D}$ if and only if 
\[
\zeta ^4=1 \quad \quad {\textup{or}} \quad \quad \zeta ^8-1=u+1=1215 u^2-190 u-1=0 .
\]
From the discussion above, we deduce that the automorphism group of the configuration consists of the diagonal elements satisfying these conditions, as well as the products of these elements with the element $\sigma_s$.  The first possibility gives rise to elements of the group $D_{4,u}$. The second possibility holds if and only if $\zeta ^8=-u=1$ and the number $1404=2^2 3^3 13$ vanishes. 
 We deduce that the group of linear automorphisms of the configuration of the inflection lines of $V_u$ coincides with the group $D_{4,u}$, with the exception of the curve $V_{-1}$ over a field of characteristic~$13$, as we always exclude fields of characteristic $2$ and $3$. Finally, in the two exceptional case, the automorphism group of the configuration of inflection lines is isomorphic to ${\rm D}_8$ and contains the linear transformation $[x, y, z] \to [y, x, z]$; since this transformation is not an automorphism of the curve $V_{-1}$, we conclude that the group of linear automorphism of $V_u$ is isomorphic to $D_{4,u}$ in all cases.

We now prove that the configurations of inflection lines of the quartics of the family $V_u$ are not conjugate by an element of $\Ps GL_3(k)$.
Suppose that $u,v$ are elements of $k$ and $T \colon \mathbb{P}^2 \to \mathbb{P}^2$ is a linear transformation such that the configurations of inflection lines of $V_u$ and $V_v$ are transformed to one another by $T$.  In particular, the conics $C_h(u)$ and $C_s(u)$ are transformed to the conics $C_h(v)$ and $C_s(v)$ by $T$.  Since the conic $C_h = C_h(u) = C_h(v)$ is independent of $u$ and since the conics $C_h , C_s(u) , C_s(v)$ are in the pencil spanned by $z^2$ and $xy$, it follows that the map $T$ stabilizes the conic $C_h$ and also the base points of the pencil.  Hence, the map $T$ is contained in the group $\overline{D}$.  Every element of $\overline{D}$ stabilized all the conics in the pencil, and therefore also the conic $C_s(u)$.  Thus the configurations that can be projectively equivalent are the ones for which the relation $27u+5 = 27v+5$ holds.  Hence the only possibility if that $u=v$ and we are done.

Finally, the configuration of inflection lines of $V$ is not conjugate to the configuration of any curve of the family $V_u$ because the hyperinflection lines of $V$ are not contained in a conic, while the ones of every quartic of the family $V_u$ are.
\qed 

\begin{Rem} \label{eccezioni}
Let $k$ be a field of characteristic $3$ or $13$; the two curves with equations 
\begin{equation} \label{ec313}
- x^4 + y^4 - z^4 - 2 x^2 y^2 - 4 x y z^2 = 0 \quad \quad {\textup{and}}  \quad \quad 
x^4 - y^4 - z^4 - 2 x^2 y^2 - 4 x y z^2 = 0 
\end{equation}
have the same configuration of inflection lines.  These curves are projectively equivalent if the field $k$ contains a square root $i$ of~$-1$: the linear transformation $\rho_i$ changes the first equation into the second one.  The case of characteristic $13$ corresponds to the exception in the statement of Theorem~\ref{aut-iso}.

In the case of a field of characteristic $3$, the smooth curves in the pencil generated by the curves in equation~\eqref{ec313} are plane quartics having in pairs the same configuration of inflection lines.
\end{Rem}

We are ready to state our main result showing that it is possible to recover any smooth plane quartic over $\co$ with at least 8 hyperinflection points from its inflection lines.  The proof of this result is an immediate consequence of Theorem~\ref{aut-iso} and Vermeulen's classification over $\mathbb{C}$ of smooth plane quartics with at least $8$ hyperinflection points.

\begin{Cor} \label{main}
Let $C_1$ and $C_2$ be smooth plane quartic defined over $\co$ and assume that $C_1$ has at least 8 hyperinflection points. If the configuration of inflection lines of $C_1$ and $C_2$ coincide, then the curves $C_1$ and $C_2$ coincide. \qed
\end{Cor}



\begin{thebibliography}{BKT}
\bibitem[APT]{APT} A. Abreu, M. Pacini and D. Testa, \emph{The general plane quartic is determined by its inflection lines}. In preparation. 
\bibitem[BKT]{BKT} F. Bogomolov, M. Korotiaev and Y. Tschinkel, \emph{A Torelli theorem for curves over finite fields}. Pure and Applied Math Quarterly, {\bf 6}, no. 1, 2010, 245--294.
\bibitem[BG]{BG} J. W. Bruce and P.J. Giblin, \emph{A stratiÞcation of the space of plane quartic curves.} Proc. London Math. Soc. {\bf 42}, 1981, 270--298. 
\bibitem[CS]{CS} L. Caporaso and E. Sernesi, \emph{Recovering plane curves from their bitangents}. Journal of  Algebraic Geometry  {\bf 12}, 2003, 225--244.
\bibitem[CEG]{CEG} C. Cumino, E. Esteves and L. Gatto, \emph{Limits of special Weierstrass points}. Int. Math. Res. Pap. {\bf 2008}, 2008, 1--65.
\bibitem[D]{Do} I. Dolgachev, \emph{Classical Algebraic Geometry: A Modern View}.  CUP 2012, available at \url{http://www.math.lsa.umich.edu/~idolga/CAG.pdf}.
\bibitem[G]{G} M. Girard, \emph{The group of Weierstrass points of a plane quartic with at least eight hyperflexes}. Mathematics of computation,  {\bf 75}, no. 255, 2006,  1561--1583. 
\bibitem[HK]{HK} A. Hefez and S. Kleiman, \emph{Notes on duality of projective varieties}.  Giornate di Geometria, 1984, Roma. Progress in Mathematics. Boston-Basel-Stuttgart: BirkhŠuser,  {\bf 60}, 1984, 143--183.
\bibitem[KK]{KK} A. Kuribayashi and K. Komiya, \emph{On Weierstrass points of non-hyperelliptic compact Riemann surfaces of genus three}. Hiroshima Math. J.  {\bf 7}, no. 3, 1977, 743--768.
\bibitem[L]{L} D. Lehavi,  \emph{Any smooth plane quartic can be reconstructed from its bitangents}.  Israel J. of Math. {\bf 146},  2005, 371--379.
\bibitem[PT]{PT} M. Pacini and D. Testa, \emph{Recovering plane curves of low degree  from their inflection lines and inflection points}. To appear in Israel J. Math. 
\bibitem[V]{V}  A. M. Vermeulen. \emph{Weierstrass points of weight two on curves of genus three}. PhD thesis, Universiteit van Amsterdam, 1983. 
\end{thebibliography}
\end{document}